\pdfoutput=1
\documentclass[11pt]{article}
\usepackage[total={6.5in,9in}, top=1in, left=1 in]{geometry}
\usepackage{caption}
\usepackage{hyperref}
\usepackage{amsmath, amsthm, amssymb}
\usepackage{graphicx}
\usepackage{epsfig,psfrag}
\usepackage{MnSymbol}
\usepackage{setspace}
\usepackage{subfigure}
\usepackage{epstopdf}
\usepackage{algorithmicx}
\usepackage{algpseudocode}
\usepackage{verbatim}
\usepackage{enumerate}
\usepackage{url}
\usepackage{color}
\usepackage{pdfsync}
\usepackage{times}
\usepackage{bm}
\usepackage{multirow}
\usepackage{booktabs}
\usepackage{xcolor}
\definecolor{Red}{rgb}{1,0,0}
\usepackage{algorithm}
\newtheorem{theorem}{Theorem}

\newtheorem{defn}[theorem]{Definition}

\newtheorem{remark}[theorem]{Remark}

\doublespace

\newcommand{\sumn}{\sum_{i=1}^{n}}
\newcommand{\prodn}{\prod_{i=1}^{n}}
\newcommand{\sumT}{\sum_{t=1}^{T}}

\newcommand{\xt}{\mathbf{x}_t}
\newcommand{\xtvec}{(x^1_t, x_t^2, \ldots, x_t^n)}
\newcommand{\xto}{\mathbf{x}_{t+1}}
\newcommand{\xit}{x^i_t}
\newcommand{\xito}{x^i_{t+1}}
\newcommand{\xterm}{\mathbf{x}_{T+1}}
\newcommand{\xiterm}{x^i_{T+1}}
\newcommand{\xinit}{\mathbf{x}_1}

\newcommand{\at}{\mathbf{a}_t}
\newcommand{\atvec}{(a^1_t, a_t^2, \ldots, a_t^n)}
\newcommand{\ait}{a^i_t}

\newcommand{\zit}{z_t^i}
\newcommand{\zt}{\mathbf{z}_t}
\newcommand{\ut}{\mathbf{u}_t}
\newcommand{\uit}{u_t^i}
\newcommand{\utvec}{(u^1_t, u_t^2, \ldots, u_t^n)}
\newcommand{\ztvec}{(z^1_t, z_t^2, \ldots, z_t^n)}

\newcommand{\dt}{D_t}
\newcommand{\dkt}{d_t^k}
\newcommand{\dto}{D_{t+1}}

\newcommand{\ft}{\mathbf{f}_t(\xt,\at)}
\newcommand{\ftvec}{(f_t^1(x_t^1,a_t^1), \ldots, f_t^n(x_t^n,a_t^n))}

\newcommand{\gt}{g_t(\xt,\at)}
\newcommand{\git}{g_t^i(\xit,\ait)}
\newcommand{\gterm}{g_{T+1}(\mathbf{x}_{T+1})}
\newcommand{\giterm}{g_{T+1}^i(x^i_{T+1})}

\newcommand{\Bt}{B_t(\xt, \at, \dt)}
\newcommand{\bb}{b_t(\dt)}
\newcommand{\Bit}{B_t^i(\xit, \ait, \dt)}

\newcommand{\lambdat}{\lambda_t}

\newcommand{\lambdatvt}{\lambda_t(\vt)}
\newcommand{\bflbda}{\boldsymbol{\lambda}}

\newcommand{\Jt}{J_t(\xt)}
\newcommand{\Jto}{J_{t+1}(\ft)}
\newcommand{\Jht}{J_t(\hist)}
\newcommand{\Jhto}{J_{t+1}(\histo)}

\newcommand{\Lte}{L_t(\xt,\vt ; \bflbda)}
\newcommand{\liute}{L_t^i(\xit,\vt ; \bflbda)}
\newcommand{\Lt}{L_t(\xt,\vt ; \bflbda)}

\newcommand{\Lto}{L_{t+1}(\xto,\vto ; \bflbda)}
\newcommand{\liuto}{L_{t+1}^i(\xito, \vto; \bflbda )}

\newcommand{\Lo}{\mathcal{L}_1}

\newcommand{\hist}{\mathbf{h}_t}
\newcommand{\histo}{\mathbf{h}_{t+1}}
\newcommand{\hterm}{\mathbf{h}_{T}}

\newcommand{\aht}{\mathbf{u}_{[t-1]}}

\newcommand{\ahtvec}{(\mathbf{a}_{0}, \mathbf{a}_1, \ldots, \mathbf{a}_{t-1})}

\newcommand{\liut}{L_t^i(\xit, \wht ; \Lambda)}

\newcommand{\E}{\mathbb{E}}
\newcommand{\Prob}{P}

\newcommand{\vt}{v_{t}}

\newcommand{\vto}{v_{t+1}}
\newcommand{\fvt}{\tilde{f}_t(\vt,\Dt)}

\newcommand{\bas}{\phi_t^r (\vto)}
\newcommand{\etw}{\eta_t^r}

\newcommand{\Dt}{D_{t+1}}

\newcommand{\yt}{\mathbf{y}_t}
\newcommand{\yit}{y^i_t}
\newcommand{\yito}{y^i_{t+1}}
\newcommand{\ytvec}{(y^1_t, y_t^2, \ldots, y_t^n)}
\newcommand{\qt}{\mathbf{q}_{t}}
\newcommand{\qit}{q^i_{t}}
\newcommand{\qito}{q^i_{t+1}}

\newcommand{\qtvec}{(q^1_{t}, q_{t}^2, \ldots, q_{t}^n)}

\newcommand{\bmin}{b^i_{\min}}
\newcommand{\bmax}{b^i_{\max}}

\newcommand{\alphait}{\alpha^i_t}
\newcommand{\betait}{\beta^i_t}

\newcommand{\ru}{r_u^i}
\newcommand{\rd}{r_d^i}
\newcommand{\liu}{\overline{l}_i}
\newcommand{\lil}{\underline{l}_i}
\newcommand{\ci}{\bar{c}_i}
\newcommand{\hi}{h_i}
\newcommand{\Hfunc}{H^i_t}
\newcommand{\Ffunc}{F^i_t}
\newcommand{\setN}{\mathcal{N}}
\newcommand{\setXt}{\mathcal{X}_t}
\newcommand{\setXit}{\mathcal{X}^i_t}
\newcommand{\setAt}{\mathcal{A}_t(\xt)}
\newcommand{\setAit}{\mathcal{A}^i_t(\xit)}
\newcommand{\setAtc}{\overline{\mathcal{A}}_t(\xt)}
\newcommand{\Aitset}{A_t^i (\yit,\ait)}
\newcommand{\Bitset}{B_t^i (\yit,\ait)}
\newcommand{\Citset}{C_t^i (\yit,\ait)}

\algblock{ParFor}{EndParFor}
\algnewcommand\algorithmicparfor{\textbf{parfor}}
\algnewcommand\algorithmicpardo{\textbf{do}}
\algnewcommand\algorithmicendparfor{\textbf{end\ parfor}}
\algrenewtext{ParFor}[1]{\algorithmicparfor\ #1\ \algorithmicpardo}
\algrenewtext{EndParFor}{\algorithmicendparfor}

\begin{document}

\begin{center}
\LARGE
A Dual Approximate Dynamic Programming Approach to Multi-stage Stochastic Unit Commitment
\end{center}

\begin{center}
Jagdish Ramakrishnan\footnote{Walmart Labs, San Bruno, CA, jramakrishnan@walmartlabs.com. This work was completed while
the author was with the Wisconsin Institute for Discovery, University of Wisconsin-Madison.} and James
Luedtke\footnote{Department of Industrial and Systems Engineering, University of Wisconsin-Madison, jim.luedtke@wisc.edu.}
\end{center}

\begin{center}
\today
\end{center}

\vspace{-10pt}

\begin{abstract}
We study the multi-stage stochastic unit commitment problem in which commitment and generation decisions can be made and
adjusted in each time period. We formulate this problem as a Markov decision process, which is ``weakly-coupled" in the
sense that if the demand constraint is relaxed, the problem decomposes into a separate, low-dimensional, Markov decision
process for each generator. We demonstrate how the dual approximate dynamic programming method of Barty, Carpentier, and
Girardeau ({\it RAIRO Operations Research}, 44:167-183, 2010) can be adapted to obtain bounds and a policy for
this problem. Previous approaches have let the Lagrange multipliers depend only on time; this can result in weak lower
bounds. Other approaches have let the multipliers depend on the entire history of past random
observations; although this provides a strong lower bound, its ability to handle a large number of sample paths or
scenarios is limited. We demonstrate how to bridge these approaches for the stochastic unit commitment problem by
letting the multipliers depend on the current observed demand. This allows a good tradeoff between strong lower bounds and good
scalability with the number of scenarios. We illustrate this approach numerically on a 168-stage stochastic unit
commitment problem, including minimum uptime, downtime, and ramping constraints.
\end{abstract}

\section{Introduction}

The unit commitment problem is an important problem in operation of power systems and has been studied extensively. Due
to the presence of both integer and continuous variables, it remains a challenging problem to solve. The basic
problem is to determine the on/off status and generation amounts of a collection of interconnected generators so that
demands are met while minimizing the total generation cost. 
An important feature of the unit commitment problem is the generator constraints, which include constraints on the
minimum and maximum generation amount, minimum and maximum number of consecutive periods the generator can be on or off
(so called ``min up/down constraints''), and bounds on the change in generation level from one period to the next
(ramping constraints). In this paper, we consider a version of the problem in which there is a single aggregate amount
of demand to be met in each time period. More complicated models also consider the transmission network and its
associated constraints, to ensure the generated electricity can be feasibly distributed to the demand locations in the
grid.

In the deterministic unit commitment problem, the future demands are  modeled as known quantities. A significant
amount of literature has focused on this problem, see e.g., \cite{FGL2009, GGZ2005, HPZ2017, WWS2013, Zhu2015}.
Stochastic formulations model the future demands as a sequence of random variables. A sequence of possible demands over time
is known as a demand scenario. As the number of demand scenarios grow, the optimization model becomes very challenging. While
our discussion is limited to handling demand uncertainties, there are a number of other uncertainties that can be
modeled in the unit commitment problem. For example, there are models that take into account generator failures \cite{WaS2006}, weather variations
\cite{UBB2016}, price spikes in the spot market \cite{HoW2011}, and availability of renewable energy \cite{ASK2017}.
There is a vast amount of literature on the stochastic unit commitment problem, see e.g., \cite{STB2004, MCP2009,
POO2011, ZWL2015, TAF2015}. A popular approach is to use a two-stage stochastic programming model \cite{Caroe1998},
where the first stage typically consists of generator on/off decisions, while the second stage consists of power
dispatch decisions (and perhaps also, on/off decision for quick-start generators) \cite{Caroe1998, Zheng2012}. These
models are appropriate when commitment decisions must be fixed for the entire planning horizon.

Multi-stage models can accurately model a longer time horizon and dependencies between time periods; this modeling
approach can be useful when generator commitment decisions may be adjusted frequently. However, with the increased
complexity, large instances of the problem (e.g., having many generators or many time periods) are very challenging to solve. This limited
scalability is due to the exponential increase in the demand scenarios with the number of stages. Note that we can view
the two-stage model to be a restriction on the multi-stage model where the generator on/off decisions are restricted to
be decided in
advance.

We begin with a Markov Decision Process formulation \cite{Put2005} of the multi-stage stochastic unit commitment problem. Direct solution of
this model is impractical for even modest-size instances, since the size of the state-space grows exponentially with the
number of generators in the system.
We therefore investigate an approximation approach that can yield a policy, along with a bound on how far it is from the
optimal policy. In particular, we apply the Dual Approximate Dynamic Programming (DADP) approach proposed in \cite{BCG2010},
which leads to an approach that decomposes the problem into a separate MDP problem for each generator in the system by relaxing the
constraints that demands must be met in each time period. The key to this approach is to allow the Lagrangian
multipliers to depend on a ``summary'' of the history of observed demands up to that time period, allowing a trade-off
(by choosing the summary) between the complexity of
solving the relaxed problem and the quality of the lower bound achieved.  This approach is related to the relaxation approach in
\cite{Haw2003, AdM2008}, but in their work the Lagrange multipliers only depend on the time period, which can result in weak lower bounds. 
On the other hand, in \cite{TaB2000} the Lagrange multipliers depend on the time period and the 
scenario of demand outcomes up to that time period. This approach can yield strong bounds, but is not practical  for
instances with many stages because the number of sample paths to a time period may grow exponentially with the number of stages. 
The DADP approach
has been applied on a small energy problem with hydraulic plants and thermal units for illustrative purposes in
\cite{BCC2010}. However, it does not capture many of the complexities in the stochastic unit commitment problem, such as min up/down and ramping constraints. 
We present a  numerical illustration on a large-scale 168-stage stochastic unit commitment problem. For bound comparisons, we generate a feasible policy and
obtain upper bounds by using the value function from the DADP approach as an approximate future value function for a
one-step lookahead policy. We show that this approach provides good lower and upper bounds for the stochastic unit
commitment problem and provides good scalability with the number of generators. 

The remainder of this paper is organized as follows. The problem formulation is given in section \ref{sec:form}. The
application of DADP to this problem is derived in section \ref{sec:dadp}, and our numerical illustration is presented in
section \ref{sec:num}.

\section{Formulation of the Stochastic Unit Commitment Problem}
\label{sec:form}


We assume there are $n$ generators and $T$ time periods indexed by $t = 1, \ldots, T$. We consider a model that
ensures that total power generation is sufficient to meet total demand in each period, but does not consider the transmission network. 
We define the random parameter $\dt$ as the electric load or demand at time $t$, which is an element of the space
$\mathcal{D}_{t} = \{\delta_t(r), r = 1, 2, \ldots, R\}$ where $\delta_t(r)$ is the $r$th possible demand realization in time period $t$. 
Demand is modeled as a Markovian process, i.e., the distribution of the random demand $\Dt$ depends on $\dt$. We define the Markovian demand distribution $\Prob_t(w \vert d)$ as the probability that $\Dt = w$, given $\dt = d$, for $w \in \mathcal{D}_{t+1}$  and $d \in \mathcal{D}_{t}$, for $t = 1, \ldots, T$. In our formulation, we assume that there is a single load to be satisfied, so the demand takes on a scalar value.   

At time $t$,
the state of the system is given by the vector $\xt = (\yt, \qt, \dt)$, where $\yt = \ytvec$ is a vector of generator
statuses, $\qt = \qtvec$ is a vector of generator production levels in the previous period, and $\dt$ is the current aggregate 
demand. The current demand $\dt$ is assumed to be observed at the end of the previous stage, and so
is included as part of the state vector, so that decisions in stage $t$ may depend on the observed value of $\dt$. 
The vector $\qt$ of previous production levels is used to enforce ramp up and down constraints for each
generator. The minimum and maximum generation levels from each generator $i=1,\ldots,n$ are denoted $b^i_{\min}$ and $b^i_{\max}$,
respectively, and hence $q_t^i \in [b^i_{\min},b^i_{\max}]$. The vector
$\yt$ keeps track of how long each generator has been on or off and is needed to enforce minimum up and down time
constraints. We also view the state as a vector of three tuples of the form $\xt = \xtvec$, where $\xit = (\yit, \qit,
\dt)$. 
Note that a generator can also represent external trading on the spot market, whether it is buying or selling
electric load for a price. In this case, the cost of producing power in such a unit would then represent the cost of
buying (positive cost) or the profit from selling (negative cost).

We denote the minimum up and down time for generator $i$ to be $\liu$ and $\lil$, respectively. Let $\yit = (\alphait, \betait)$, where $\alphait \in [0, \ldots,\liu]$ represents the number of periods the generator has been on, and $\betait \in [0, \ldots,\lil]$ represents the number of periods the generator has been off. Either $\alphait$ or $\betait$ must be zero at any point in time, but they cannot be zero simultaneously. If the generator has been on for more than $\liu$ time periods, then $(\alphait, \betait) = (\liu, 0)$, meaning the generator can be turned off. Similarly, if the generator has been off for more than $\lil$ time periods, then $(\alphait, \betait) = (0, \lil)$, meaning the generator can be turned on. If the generator is on and must remain on for some more time, $\alphait$ will be a positive integer but strictly less than $\liu$ and $\betait$ will be zero, and vice versa if the generator is off. For initialization purposes, we could set $y_1^i = (0, \lil)$, which would mean the generator has remained off for long enough that it can be turned on.

In summary, the state space in period $t$ for each generator $i=1,\ldots,n$ is defined as
\begin{equation*}
\setXit = \{ (\yit, \qit, d) : \yit = (\alphait, \betait),  \alphait \in [0, \ldots,\liu], \betait \in [0, \ldots,\lil],
\qit \in [b^i_{\min},b^i_{\max}], d \in \mathcal{D}_{t} \},
\end{equation*}
for $t = 2, \ldots, T$, and $\mathcal{X}_1^i = \{ (y_1^i, q_1^i, d) : y_1^i = (0, \lil), q_1^i = 0, d = 0 \}$.
The overall state $\xt$ is a member of the state space $\setXt$, which is the
Cartesian product of the individual state spaces $\setXit$, i.e., $\xt \in \setXt = \times_{i=1}^n \setXit$. 

The actions at time $t$ are denoted by the vector $\at = (\zt, \ut)$, where $\zt = \ztvec$ is a vector of generator
production levels, and $\ut = \utvec$ is a vector of binary generator on/off decisions for the next stage. Here, $\uit = 0$ means generator $i$ is off, and $\uit = 1$
means the generator is on. We assume that the on/off decisions are made for the next period, whereas the generation
decisions are made for the current period, after observing the current demand $D_t$. Thus, at time $t$, the generation decisions $\zt$ are for the current period $t$, and the commitment decisions $\ut$ are for period $t + 1$.

For each generator $i$, we enforce the following: minimum and maximum production level bounds, minimum up and down time
constraints, and ramp up and down production level constraints.  The minimum and maximum production levels
 for generator $i$ are enforced with the following constraint:
\begin{equation}
\label{eq:min_max}
\underline{u}(\yit) b^i_{\min} \leq \zit \leq \underline{u}(\yit) b^i_{\max} ,
\end{equation}
where $\underline{u}(\yit)$ is the applied commitment decision $u^i_{t-1}$ which equals $1$ when $\yit = (j, 0)$ for $j
= 1, \ldots, \liu$ and $0$ otherwise. In the above constraint, if the previous commitment decision was
$\underline{u}(\yit) = 0$, the production level is set to $0$; otherwise, the production level remains between its
minimum and maximum levels. We enforce minimum up and down constraints by requiring:
\begin{equation}
\label{eq:min_up_down}
\underline{I}(\yit) \leq \uit \leq \overline{I}(\yit),
\end{equation}
where $\underline{I}(\yit)$ equals 1 if $\yit = (j, 0)$ for $j = 1, \ldots, \liu - 1$ and 0 otherwise, and $\underline{I}(\yit)$ equals 0 if $\yit = (0, j)$ for $j = 1, \ldots, \lil - 1$ and 1 otherwise. For ramp up and down constraints, we enforce:
\begin{equation}
\label{eq:ramp_up_down}
\qit - \rd - (1 - \underline{u}(\yit)) b^i_{\min} \leq \zit \leq \qit + \ru + \overline{w}(\yit) b^i_{\min},
\end{equation}
where $\underline{u}(\yit)$ is the applied commitment decision $u^i_{t-1}$ as before, $\overline{w}(\yit)$ is a
``turn on'' indicator that is $1$ if $\yit = (1, 0)$ and 0 otherwise, and $\ru$ and $\rd$ are ramp up and down amounts for generator $i$, respectively. 

At each time period, we enforce a linking constraint that ensures the sum of the production levels from each generator satisfies the demands observed: 
\begin{equation}
\label{eq:link}
\sumn z_t^i = \dt.
\end{equation}
Since we assume $D_1 = 0$ and all the generators are initially off in the first stage, the above constraint would mean $z_1^i = 0$ for all $i$. This model assumes total generation should \emph{exactly} meet the load. An alternate constraint could ensure total
generation to \emph{at least} meet the load; we discuss minor changes in the solution approach if this were modeled in a
later section. This model may be extended to allow $z_t^i$, and $\dt$ to be vectors, e.g., if we have multiple electric loads, although
we focus on the scalar case. 

In summary, the control space for each generator $i=1,\ldots,n$ is defined as
\begin{equation*}
\setAit = \{ (\zit, \uit) : (\ref{eq:min_max}), (\ref{eq:min_up_down}), (\ref{eq:ramp_up_down}), \uit \in \{0, 1\} \}.
\end{equation*}
Then, the overall feasible action space in stage $t$ is defined as 
\begin{equation*}
\setAtc = \{ \at = \atvec \in  \times_{i=1}^n \setAit : \sumn z_t^i = \dt \}.
\end{equation*}
To initialize the model, in the first period, $t=1$, we assume we only make commitment decisions.  Thus, we assume $D_1 = 0$ and hence $z_1^i = 0$ for all $i$. 

We now define the state update equations. For the state $\yit$, we have:
\begin{equation*}
\yito =
\begin{cases}
    (\alphait + 1, 0), &\text{if $0 < \alphait < \liu$}\\
    (0, \betait + 1), &\text{if $0 < \betait < \lil$}\\
    (\liu, 0), &\text{if $\alphait = \liu$, $\uit = 1$}\\
    (0, 1), &\text{if $\alphait = \liu$, $\uit = 0$}\\
    (0, \lil), &\text{if $\betait = \lil$, $\uit = 0$}\\
    (1, 0), &\text{if $\betait = \lil$, $\uit = 1$.}
\end{cases}
\end{equation*}
The state update equations for $\qit$ representing previous production levels is $\qito = z_t^i$. The overall update equation $\xto = \ft = \ftvec$ represents all of the above update equations taken together.

At time $t$ for generator $i$, the cost $\git$ is the total expected generation cost. We define $\ci$ to be the no load
cost (fixed cost for generator being on), $\hi$ to be a fixed cost for turning on generator $i$ when it is off, and $\Ffunc (z)$ to be the generation cost of producing $z$. We model $\Ffunc (z)$ as a piecewise linear function of $z$. For generator $i$, we evaluate price-quantity bids over a grid and denote these points $(b^i_k, c^i_k)$, for $k = 0, \ldots, K_i$, where $b^i_k$ is the $k$th generation level and $c^i_k$ is the cost associated with it. These are the breakpoints of the piecewise linear generation cost function. Note that based on the previous notation, we have $\bmin = b^i_0$ and $\bmax = b^i_{K_i}$, for $i = 1, \ldots, n$. The cost for time period $t$ is incurred after implementing the controls $\at$. The startup cost $\Hfunc (\yit,\uit)$ is defined as
\begin{equation*}
\Hfunc (\yit,\uit) =
\begin{cases}
	\hi,	&\text{if $\yit = (0, \lil), \uit = 1$}\\
        0,	    &\text{otherwise.}
\end{cases}
\end{equation*}
Now we define the total cost incurred per time period as the sum of the start up cost, no load cost, and the generation cost:
\begin{equation*}
\git = \Hfunc (\yit, \uit) + \ci \uit + \Ffunc (z_t^i),
\end{equation*}
for $t = 1, \ldots, T-1$. For the last period $T$, the on/off decisions are irrelevant since commitment decisions are determined for the next stage; thus, only the cost associated with production level decisions are incurred:
\begin{equation*}
g_T^i (x_T^i, a_T^i) = F_T^i (z_T^i).
\end{equation*}

The overall cost incurred at time $t$ is the sum of the individual costs, i.e., $\gt = \sumn \git$. Thus, we formulate the stochastic unit commitment problem as
\begin{equation*}
\underset{\pi}{\min} \; \E \left[ \sumT \gt \right],
\end{equation*}
where $\pi = \{ (\bm{\zeta}_0, \bm{\mu}_0), \ldots, (\bm{\zeta}_{T-1}, \bm{\mu}_{T-1}) \}$ represents an admissible policy, where $(\bm{\zeta}_t, \bm{\mu}_t)$ maps the state $\xt$ into actions $(\zt, \ut) = (\bm{\zeta}_t(\xt), \bm{\mu}_t(\xt))$ such that $(\bm{\zeta}_t(\xt), \bm{\mu}_t(\xt)) \in \setAtc$ for all $\xt \in \setXt$. Note that we have not included a terminal cost associated with being in a potential undesirable state after applying the sequence of decisions; this would be a straightforward addition to the cost, e.g., $\E \left[ \gterm \right]$. 
If we define $\Jt$ to be the minimum expected cost-to-go when the system is in state $\xt \in \setXt$, then $\Jt$
satisfies the dynamic programming (DP) recursion 
\begin{equation}
\label{eq:DP}
\Jt = \underset{\at \in \setAtc}{\min} \bigg\{  \E \left[ \gt +  \Jto \vert \xt, \at \right] \bigg\},
\end{equation}
for $t = 1, \ldots, T$, where $J_{T+1}(\xterm) = 0$, and the expectation is taken with respect to the probability
distribution $\Prob_t(\Dt \vert \dt)$.

The notation described in this section is summarized in Table \ref{table:notation}.

\begin{table}[ht]
\caption{Notation} 
\centering 
\begin{tabular}{c l} 
\hline\hline 
Constant & Description \\ [0.5ex] 
\hline 
$\yit$ & state variable indicating status of $i$th generator at time $t$ \\
$\qit$ & state variable indicating generator $i$'s previous production level at time $t$ \\
$\dt$ & observed demand at time $t$ \\
$\zit$ & production level decision of $i$th generator at time $t$ \\
$\uit$ & binary on/off decision of $i$th generator at time $t$ \\
$\ru$ & maximum ramp up amount for generator $i$ \\
$\rd$ & maximum ramp down amount for generator $i$ \\
$\liu$ & minimum up time for generator $i$ (minimum time generator must stay on after being turned on) \\
$\lil$ & minimum down time for generator $i$ (minimum time generator must stay off after being turned off) \\
$\ci$ & no load cost for generator $i$ (fixed cost for generator being on) \\
$\hi$ & turn on cost for generator $i$ (additional cost for turning on generator when it is off) \\ [1ex] 
\hline 
\end{tabular}
\label{table:notation} 
\end{table}


\section{Dual Approximate Dynamic Programming Approach}
\label{sec:dadp}

We now describe how we adapt the dual approximate dynamic programming approach \cite{BCG2010, Gir2010, Lec2014} to obtain a policy and optimality bound for the stochastic unit commitment problem. While the DADP approach has been applied previously to a hydraulic valley example and simple small-scale power management problem, the problem did not have any integer variables and did not capture the complexities including min / up down and ramping constraints. We show for the first time its effectiveness on a large-scale stochastic unit commitment problem.

In time period $t$, an exact approach using the original DP recursion (\ref{eq:DP}) would result in a total number of
states of $| \setXt | = \prod_{i=1}^n | \setXit |$. Even with a relatively small number of states for each subproblem,
this solution approach would quickly become computational intractable because of the number of states growing
exponentially with the number of generators. In the DADP approach, a Lagrangian relaxation approach is used and the
resulting subproblems are solved independently. This approach instead solves a problem with $| \setXit |$ for
each generator $i=1,\ldots, n$. Having solved the relaxed problem, the approach provides a lower bound on the original optimal
objective. We can obtain a primal policy (and hence an upper bound) by using a one step lookahead policy by using an
approximate value function derived from the Lagrangian relaxation solution.

The main idea of the DADP approach is to introduce an additional state:
\begin{equation}
\label{eq:vt}
\vt = \tilde{f}_t(v_{t-1},\dt),
\end{equation}
which summarizes the exogenous information process $D_1,D_2, \ldots, D_{t}$, for $t = 2, \ldots, T$, with an initial
state $v_1 = D_1 = 0$. 
For each $t=2,\ldots,T$, we assume $\vt$ lies within a finite set of values, denoted by the set $\mathcal{V}_t$.
For the DADP approach to be computationally practical, the size of $\mathcal{V}_t$
must not be too large.  
We let the Lagrange multipliers depend on $\vt$, and hence define $\lambda_t:\mathcal{V}_t \rightarrow \mathbb{R}$, for $t = 1, \ldots, T$.
Here, if there are multiple linking constraints, $\lambda_t$ would be vector-valued, with each element
representing multipliers for each linking constraint. We assume that knowing $\vt$ is sufficient to know the distribution
of $D_r$ for any $r > t$. This is trivially satisfied if the random demands are stage-wise independent. More generally,
if knowing $\vt$ implies we know $\dt$ (e.g., $\vt$ may be a vector containing $\dt$ as one component), then this is
implied by the Markovian assumption. 

We let $\bflbda = [\lambda_t]_{t=1}^T$, be the collection of all Lagrangian multipliers. For a fixed $\bflbda$, the Lagrangian problem is:
\begin{equation}
\label{eq:L_lambda}
\underset{\pi}{\min} \; \E \left\{ \sumT \left[ \gt + \lambdatvt^\top \left(\sumn z_t^i - \dt\right) \right] \right\},
\end{equation}
where $\pi$ represents the class of admissible policies over the feasible control space $\setAt$. Now, the Lagrangian recursion is:
\begin{equation*}
L_{T+1}(\xterm, v_{T+1}; \bflbda ) = 0,
\end{equation*}
\begin{equation}
\label{eq:Leta_recursion}
\Lte = \underset{\at \in \setAt}{\min} \E \left[ \gt + \lambdatvt^\top \left(\sumn z_t^i - \dt\right) +  \Lto \; \vert \; \xt, \vt, \at \right]
\end{equation}
for $t = 1, \ldots, T$.

\begin{remark}
The representation (\ref{eq:vt}) is very general; note that by letting $\vt = [v_{t-1}, \dt]$, we have a multiplier for every sequence in the exogenous demand process.
\end{remark}

\begin{remark}
Since (\ref{eq:link}) is an equality constraint, we let $\lambda_t$ to be a free variable. However, we could have allowed inequality linking constraints as well, in which case we would have ensured $\lambda_t \geq 0$.
\end{remark}

\subsection{Decomposition and Structural Properties}

We present the key results that are needed for applying the DADP approach to this problem. See \cite{Gir2010, Lec2014} for further results.

The following result shows that under this representation, the Lagrangian problem decomposes into $n$ individual
subproblems. In the theorem, the notation  $\E [ \lambda_r(v_r)D_r \; \vert \; v_t ]$  for $1 \leq t \leq r \leq T$ represents the expected value
of $\lambda_r(v_r)D_r$ given that the state of the demand process in stage $t$ is $v_t$, where the expectation is taken with
respect to the random outcomes $D_{t+1},\ldots,D_r$. Note that for $r=t$, this term is simply $\lambda_t(v_t)D_t$.

\begin{theorem}
\label{thm:Lg}
The Lagrangian recursion decouples as follows:
\begin{equation*}
\Lte = \sumn \liute - \sum_{r=t}^T  \E [ \lambda_r(v_r)D_r \; \vert \;  v_t ],
\end{equation*}
for $t=1,\ldots,T$, where 
\begin{equation*}
L_{T+1}^i(\xiterm,v_{T+1}; \bflbda) = 0,
\end{equation*}
and for $t=T,\ldots,1$
\begin{equation}
\liute = \underset{\ait \in \setAit}{\min} \E \left[ \git + \lambdatvt  z_t^i  + \liuto \; \vert \; \xit, \vt, \ait \right].
\label{eq:Lite}
\end{equation}
\end{theorem}

\begin{proof}
We proceed by induction. For the base case, we have by definition
\begin{equation*}
L_{T+1}(\xterm, v_{T+1} ; \bflbda) = 0 = \sumn L_{T+1}^i(\xiterm,v_{T+1} ; \bflbda).
\end{equation*}
Now, assume the statement in the theorem holds for time $t+1$. Then, we have
\begin{align*}
\Lte &= \underset{\at \in \setAt}{\min} \E \Bigg[ \gt + \lambdatvt  \left(\sumn z_t^i - \dt\right) +  \Lto \; \Bigg\vert \; \xt, \vt, \at \Bigg] \\
&= \underset{\at \in \setAt}{\min}\E \Bigg[  \sumn \git + \lambdatvt  \left(\sumn z_t^i - \dt\right) +   \sumn \liuto  \\
& \qquad\qquad\qquad - \sum_{r=t+1}^T \E [ \lambda_r(v_r)D_r \; \vert \; v_{t+1} ] \; \Bigg\vert \; \xt, \vt, \at \Bigg] \\
&= \underset{\at \in \setAt}{\min} \sumn \E \left[ \git + \lambdatvt  z_t^i +  \liuto \; \vert \; \xit, \vt, \ait \right] \\ 
& \qquad\qquad - \sum_{r=t}^T \E [ \lambda_r(v_r)D_r \; \vert \; \vt ] \\
&= \sumn \liute - \sum_{r=t}^T \E [ \lambda_r(v_r)D_r \; \vert \;  v_{t} ],
\end{align*}
where $\liute$ satisfies (\ref{eq:Lite}), as desired.
\end{proof}

In particular, Theorem \ref{thm:Lg} implies that
\begin{equation}
\label{eq:Lg_decompose}
L_1(\xinit,v_1 ; \bflbda) = \sumn L_1^i(x_1^i,v_1 ; \bflbda) - \sum_{t=1}^T \E [ \lambda_t(v_t)D_t \; \vert \;
v_1 ]. 
\end{equation}
The importance of Theorem \ref{thm:Lg} is that, for fixed $\bflbda$, $L_1(\xinit,v_1 ; \bflbda)$ can be evaluated by solving $n$ independent
Markov decision problems, each with a relatively small state space. The term
$\sum_{t=1}^T \E [ \lambda_t(v_t)D_t \; \vert \;  v_1 ]$ is independent of the decision process, and so can be estimated
via simulation. 

We next present two important structural properties of $\mathcal{L}$.

\begin{theorem}
We have that
\begin{enumerate}
\item $L_1(\xinit,v_1 ; \bflbda) \leq J_1(\xinit)$ for all $\bflbda$.
\item $L_1(\xinit,v_1 ; \bflbda)$ is concave function of $\bflbda$. 
\end{enumerate}
\label{prop2}
\end{theorem}
\begin{proof}
These results follow from standard Lagrangian theory \cite{Ber2016, BoV2004}. For any feasible policy, we have $\sumn z_t^i = \dt$ for all $\dt$, $t = 1, \ldots, T$. Thus, by definition
\begin{align*}
L_1(\xinit,v_1 ; \bflbda) &= \underset{\pi}{\min} \; \E \left[ \sumT \gt + \lambdatvt  \left(\sumn z_t^i - \dt\right) \right], \\
&\leq \E \left[ \sumT g_t(\xt, (\bm{\zeta}_t(\xt,\vt), \bm{\mu}_t(\xt,\vt))) \right],
\end{align*}
for a feasible policy $(\bm{\zeta}_t, \bm{\mu}_t)$ where $\bm{\zeta}_t$ is the policy associated with production levels and $\bm{\mu}_t$ is associated with the commitment decisions. Since the above is true for any feasible policy, it also holds for an optimal policy $(\bm{\zeta}_t^*(\xt,\vt), \bm{\mu}_t^*(\xt,\vt))$. We now have
\begin{equation*}
L_1(\xinit,v_1 ; \bflbda) \leq \E \left[ \sumT g_t(\xt, (\bm{\zeta}_t^*(\xt,\vt), \bm{\mu}_t^*(\xt,\vt))) \right] = J_1(\xinit).
\end{equation*}
For the second claim, we proceed by induction, and use the recursive definition of $L_1(\xinit,v_1 ; \bflbda)$ given in
(\ref{eq:Leta_recursion}). For the base case, we see that $L_{T+1}(\xterm, v_{T+1}; \bflbda )$ is clearly concave in $\bflbda$. Now, suppose $\Lto$ is concave in $\bflbda$. Then, the expected value term in (\ref{eq:Leta_recursion}) is a concave function of $\bflbda$. $\Lte$ is concave because it is a minimum of concave functions of $\bflbda$.
\end{proof}

We find the best lower bound and thus, maximize $L_1(\xinit,v_1 ; \bflbda)$ over $\bflbda$. We define
\begin{equation*}
\Lo = \underset{\bflbda}{\max} \; L_1(\xinit,v_1 ; \bflbda).
\end{equation*}
From Theorem \ref{prop2}, it follows that $\Lo \leq J_1(\xinit)$.

Because $L_1(\xinit,v_1 ; \bflbda)$ is concave in $\bflbda$, but not necessarily smooth, it can be maximized using supergradient-based methods.
In general, it is difficult to determine the supergradient exactly. However, we can obtain an unbiased  stochastic
estimator of a supergradient using sampling. The following theorem shows how to compute an unbiased estimator of a supergradient of $L_1(\xinit,v_1 ; \bflbda)$.

\begin{theorem}
Suppose $\pi = \{(\bm{\zeta}_1, \bm{\mu}_1), \ldots,
(\bm{\zeta}_{T}, \bm{\mu}_{T})\}$ is the optimal policy for the 
the Lagrangian relaxation problem $L_1(\xinit,v_1 ; \bflbda)$, where $\bm{\zeta}_t$ is associated with generator
production levels and $\bm{\mu}_t$ is associated with commitment decisions. Here, each subproblem $i$ has its policy
$\pi^i = \{(\zeta^i_1, \mu^i_1), \ldots, (\zeta^i_{T},
\mu^i_{T})\}$. An unbiased estimator of a supergradient of $\mathcal{L}$ at $\bflbda$ is
\begin{equation*}
\Bigl[\sumn \zeta^i_t(\xit,\vt) - \dt\Bigr]_{t=1}^T .
\end{equation*}
\label{thm:subgradient}
\end{theorem}
\begin{proof}
For any $\hat{\bflbda}$, we have
\begin{align*}
L_1(\xinit,v_1 ; \hat{\bflbda}) &\leq \E \sumT \Bigg[ g_t(\xt, (\bm{\zeta}_t(\xt,\vt), \bm{\mu}_t(\xt,\vt))) +
\hat{\lambda}_t^\top \left(\sumn \zeta^i_t(\xit,\vt) - \dt\right) \Bigg] \\
&= \E \sumT \Bigg[ g_t(\xt, (\bm{\zeta}_t(\xt,\vt), \bm{\mu}_t(\xt,\vt))) + \lambda_t^\top \left(\sumn \zeta^i_t(\xit,\vt) - \dt \right) \\
& \quad\quad\quad + (\hat{\lambda}_t - \lambda_t)^\top \E \left(\sumn \zeta^i_t(\xit,\vt) - \dt\right) \Bigg] \\
&= L_1(\xinit,v_1 ; \bflbda)+ \sumT (\hat{\lambda}_t - \lambda_t)^\top \E \left(\sumn \zeta^i_t(\xit,\vt) - \dt\right),
\end{align*}
where the first inequality follows because $\pi$ is a feasible, but not necessarily optimal, policy for
$\hat{\bflbda}$. It follows that $\bigl[\E (\sumn \zeta^i_t(\xit,\vt) - \dt)\bigr]_{t=1}^T$ is a supergradient at $\bflbda$.
\end{proof}

The importance of Theorem \ref{thm:Lg} is that a {\it stochastic} supergradient method \cite{BoM2014, Ber2015, Ber2016} can then be applied to maximize
$L_1(\xinit,v_1 ; \bflbda)$. For a fixed $\bflbda$, we can solve the Lagrangian recursion via a decoupled approach given in Theorem \ref{thm:Lg}. 
Then, for any simulated sample path of the random variables $D_t$, for $t=1, \ldots, T$, the vector of demand violations $[\sumn
\zeta^i_t(\xit,\vt) - \dt]_{t=1}^T$, is an unbiased estimate of a supergradient of $L_1$ at $\bflbda$. 
To get a better (reduced variance) estimate of the supergradient, we can use batch gradient averages, i.e., simulate
many sample paths and average the demand violations to obtain a supergradient estimate. 

\subsection{Implementation of the DADP Approach}


We solve the decoupled MDP subproblems using the recursion equation given in (\ref{eq:Lite}) using standard dynamic
programming. The only detail we need to deal with is that, as stated, the state space is not finite, so we cannot
directly enumerate all states when calculating the value function. To address this, for generator $i$, and for each
 possible value of $\yit$ and $D_{t} \in \mathcal{D}_{t}$, we discretize $q_{t}^i$ at the same points as for the cost
 function $\git$, i.e.,  $b_k^i$, for $k = 1, \ldots, K_i$, for every $t$. We denote the discretized version of the set $\setXt$ by $\tilde{\mathcal{X}}_t$. 
 
 
  
 Because both the piecewise-linear cost functions have the same set of break points in each time period, these break
 points are inherited by the value function, and hence the minimization in (\ref{eq:Lite}) always has a
 solution at one of those break points or at the bounds. To avoid having a solution at the bounds, we relax the bound
 constraints to allow one discretization point above and below the ramping bounds. Thus, although we restrict the
 production levels (and state variables) to lie at one of the break points, the value we obtain from the algorithm will
 still be a lower bound on the optimal value. Suppose $\tilde{\mathcal{A}}_t^i(\xit)$ is the discretized version of the set $\setAit$, where $\zit$ only takes on the values $b_k^i$, for $k = 1, \ldots, K_i$, or $b_0^i = 0$ if the generator is off, and the set is modified so that the ramping constraint (\ref{eq:ramp_up_down}) is relaxed to include one discretization point less than or equal to the lower bound and one point greater than or equal to the upper bound. Thus, we solve the discretized version of (\ref{eq:Lite}) by simply taking a pointwise minimum over the objective evaluated at their break points.
The overall DADP algorithm is given in pseudocode in the Algorithm \ref{alg} box below. The ``parfor'' loops indicate loops that can be parallelized. 


\begin{algorithm}
\caption{DADP Algorithm}\label{alg}


\begin{algorithmic}
\State \textbf{Parameters:} $\rho, \eta, \text{maxIters}, \text{batchSize}$
\State Initialize $\lambda^1_t(v_1)$, for $t = 1, \ldots, T$
\For{($k=1;~k\gets k+1;~k\leq \text{maxIters}$)}
\ParFor{($i=1;~i\gets i+1;~i\leq n$)}
\For{($t=T;~t\gets t-1;~t > 0$)}
\For{($(x, v) = ((l, b_k^i, d), v) \in (\tilde{\mathcal{X}}_t, \mathcal{V}_t)$)}

\State $L_t^i(x, v; \mathbf{\lambda}) = \underset{\ait \in \tilde{\mathcal{A}}_t^i(x)}{\min} \big[ g_t^i(x, \ait) + \lambda_t^k(v)  z_t^i  + \sum_{r \in \mathcal{D}_{t+1}} P_t(r | d)  L_{t+1}^i(f_t^i(x, \ait), \tilde{f}_t(v, r); \bflbda) \big]$ 
\State and let $(\zeta^i_t(x, v), \mu^i_t(x, v))$ be an optimal action. 
\EndFor
\EndFor
\EndParFor
\State $g_t(\vt) = 0$
\ParFor{($j=1;~j\gets j+1;~j\leq \text{batchSize}$)}
\For{($t=1;~t\gets t+1;~t\leq T$)}
\State Simulate $D_t$ 
\State Update states $\xt = \mathbf{f}_t(\mathbf{x}_{t-1}, (\bm{\zeta}_{t-1}(\mathbf{x}_{t-1},v_{t-1}), \bm{\mu}_{t-1}(\mathbf{x}_{t-1},v_{t-1})))$, $\vt = \tilde{f}_t(v_{t-1}, D_t)$
\State $g_t(\vt) = g_t(\vt) + (\sumn \zeta^i_t(\xit,\vt) - \dt$) / $\text{batchSize}$
\EndFor
\EndParFor
\State $\lambda^k_t(v_1) \gets \lambda^{k-1}_t(v_1) + \rho \eta^k g_t(\vt)$, where $\rho$ and $0 < \eta < 1$ are step size parameters
\EndFor
\State $\Lo =  \sumn L_1^i(x_1^i,v_1 ; \bflbda) - \sum_{t=1}^T \E [ \lambda_t(v_t)D_t \; \vert \;
v_1 ] $, where initial states $x_1^i = ((0, \lil), 0, 0)$ and $v_1 = 0$. \\
\Return $\Lo, \; (\bm{\zeta}_{t}(\xt,\vt), \bm{\mu}_{t}(\xt,\vt))$, for $t = 1, \ldots, T$
\end{algorithmic}
\end{algorithm}

\subsection{One-step Lookahead Policy and an Upper Bound}
\label{subsection:upper}

Solving the Lagrangian relaxation (\ref{eq:L_lambda}) gives a lower bound to the optimal value. However, the policy obtained from solving this problem via a DADP approach is not guaranteed to be feasible. One way to obtain a feasible policy is to approximate the future value function with the relaxed value functions, i.e., with $\Lte$, for $t = 1, \ldots, T$. We use a one-step lookahead policy \cite{Ber2005}. 
We first use the DADP approach to optimize for $\bflbda$ and solve $n$ independent MDP and their associated relaxed value functions $\Lte$, for $t = 1, \ldots, T$.
Then, we simulate a sample path $D_1, \ldots, D_{T}$. For time $t=1, \ldots, T$, we solve the following:
\begin{equation}
\label{eq:lookahead}
\hat{J}_t(\xt) = \underset{\at \in \setAtc}{\min} \bigg\{  \E \left[ \gt +  \Lto \vert \xt, \at \right] \bigg\},
\end{equation}
for the state $\xt$ which determines the generation decisions $\zit$ for the current period and the commitment
decisions $\uit$ that will determine the available generators in the next period. Problem \eqref{eq:lookahead} can be
formulated as a deterministic mixed-integer program with decision variables $(\zit, \uit)$ for $i=1,\ldots,n$ and with
constraints (\ref{eq:min_max})-(\ref{eq:link}). Note that in this policy we enforce the correct lower and upper bounds
implied by ramping, and hence the policy is feasible. The piecewise linear cost function and relaxed value functions $\Lto$ are modeled with turn on variables and non-negative variables, similar to what is done for the perfect information MIP in section \ref{subsec:per_info}. Since this lookahead  policy is a feasible policy, we can obtain a stochastic upper bound by generating a large number of samples and averaging the resulting overall costs from applying the above policy.

The Lagrangian dual problem may be solved offline once to generate lower bounds, and the obtained value function
approximation may be used for the one-step lookahead policy through the whole time horizon. However, in cases where the
demand model does not exactly follow the assumed distribution, it may be advantageous to re-optimize the Lagrangian dual
every so often with updated demand information. 

\subsection{Discussion}

The DADP approach is an extension of the approach by Adelman and Mersereau \cite{AdM2008}, which describes using a state independent multiplier $\lambda_t$ for Lagrangian decomposition. Their paper shows that using a state dependent multiplier, e.g., $\lambda_t (\xt)$, does not result in decomposition. However, in the DADP approach, dependence on only the exogenous demand history or some function of it, e.g., $v_t$, results in the desirable decomposition property.

Our model formulation for the stochastic unit commitment problem closely follows the one in Takriti et al
\cite{TKW2000}. Their approach also uses Lagrangian decomposition, but the multipliers depend on the scenario. This
dependence results in an exponential increase in the number of scenarios with the decision stages, and the solution
approach in \cite{TKW2000} quickly becomes intractable, or requires an overly coarse representation of the stochastic
process. For example, if we assume there are 10 demand scenarios for every hour of the week (168 periods), we would have a total of $10^{168}$ multipliers, each associated with a scenario. Thus, this limits the approach to a relatively small number of decision stages.

The scenario-based approach is also special case of the DADP approach when the state $\vto$ consists of the full demand history:
\begin{equation*}
v_t = [D_1, \ldots, D_{t-1}, D_t].
\end{equation*}
As mentioned earlier, this representation results in a huge state space and is therefore limited to a small number of
stages. The key in the DADP approach is the selection of a good ``summary function'' $f_t(v_{t-1}, D_{t})$ for
summarizing the demand process up to stage $t$.  This requires finding a tradeoff between letting $\vt$ represent the
full demand history and ignoring the history completely (which is the state independent multiplier case in Adelman and Mersereau \cite{AdM2008}).

\section{Numerical Illustration}
\label{sec:num}

All implementations and problem instances can be obtained at \href{https://github.com/jramak/dual-adp-suc}{https://github.com/jramak/dual-adp-suc}. 
\subsection{Problem Data}

We use generator data from the FERC eLibrary Docket Number AD10-12, ACCNNUM 20120222-4012. This included min up/down
times, ramp up/down amounts, no load costs, turn on costs, and up to 10 pairs of price-quantity bids. Of the 1011
generators, we randomly selected generators for our 15, 30, and 50 generator test cases. We obtained 2013 hourly demand
data from the PJM Interconnection ISO, which is the regional transmission organization for the eastern electricity
market. In order to model realistic demand fluctuations, we averaged out demand for each of the 168 hours of the week
and normalized it by the maximum demand (see Figure \ref{fig:demand}). Note that we have $T=169$ because in the first period $D_1 = 0$ and $z_1^i = 0$ for all $i$, and only the on/off decisions $\mathbf{u}_1$ for the next (second) stage are determined. To create a reasonable problem for each test
case, we scaled this normalized demand by a percentage of the maximum combined generation level of all generators,
\emph{TotCap}, in the test set. We denote $\mu$ to be the percentage of the maximum generation level, e.g., the scaling
factor was $\mu*$\emph{TotCap}. We assume demand is independent between time periods and sample 10 possible demand
scenarios for each period. To generate demand scenarios for each time $t$, we evaluated 10 points using the
Legendre-Gauss Quadrature in the interval $[d_{scaled} - 4 \sigma d_{scaled}, d_{scaled} + 4 \sigma d_{scaled}]$, where
$d_{scaled}$ is the scaled mean demand and $\sigma$ represents the percentage of the variation in the mean demand. With
the choices  $\mu = 0.4, 0.6, 0.8$ and $\sigma = 0.15, 0.20, 0.25$, we had a total of 9 instances for each test case.
For modeling the generation cost, we evaluated the price-quantity bids over a uniform grid between the minimum and
maximum generation levels (i.e., $\bmin$ and $\bmax$) using 50 points. 

\begin{figure}
\begin{center}
\includegraphics[width=4.0in]{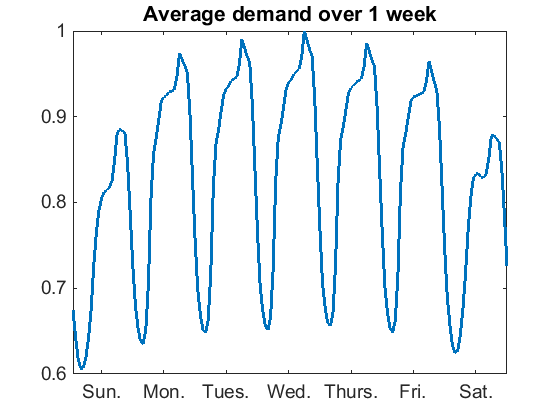}
\end{center}
\caption{Average hourly demand from the 2013 PJM Interconnection ISO normalized by the maximum average demand.}
\label{fig:demand}
\end{figure}

\subsection{DADP-based Bounds}

For the DADP approach, we use $v_t = D_t$, for $t =  1, \ldots, T$, which with 10 possible demand values per period increases the state space by 10 compared to the state independent case. Due to having $v_t = D_t$ and the stagewise independence assumption for our numerical example, we calculate for a fixed $\bflbda$ the expected value term exactly in the Lagrangian function (\ref{eq:Lg_decompose}):
\begin{equation*}
 \sum_{t=1}^T \E [ \lambda_t(v_t)D_t \; \vert \; v_1 ] = \sum_{t=1}^T \E [ \lambda_t(D_t)D_t ] = \sum_{t=1}^T \sum_{\delta \in \mathcal{D}_t} [ P_t(\delta) \; \lambda_t(\delta) \; \delta ]
\end{equation*}
Because of the exact calculation of the expected value term, the lower bounds produced by the DADP approach are deterministic.

For the stochastic supergradient method, we used a batch of 1000 sample paths to estimate the supergradient $g_t$ at each
step. For the supergradient estimation, we used a serial implementation, but this could be easily parallelized. The $k$th iteration is defined as
\begin{equation*}
\lambda_{t}^{k}(D_{t}) = \lambda_t^{k-1}(D_t) + \rho \eta^k g_t,
\end{equation*}
for $t = 1, \ldots, T$, where $\rho$ and $0 < \eta < 1$ are step size parameters. For large enough $\rho$ and $\eta$, the supergradient method is
guaranteed to converge although larger values can result in slower convergence. We used $\rho = 50 /(
\mu*\text{\emph{TotCap}})$, $\eta = 0.99$, and set 250 as the maximum number of iterations. The parameter choices we chose resulted in apparent convergence and provided good enough solution quality and times, and we did not optimize the choice of the parameters further. Similar to the approach in \cite{TKW2000}, we initialize the multipliers $\lambda_t(D_t)$ by approximating the cost in period $t$ for generator $i$ by a linear function and solving the approximate problem. The slope of the linearized cost function at time $t$ for generator $i$ is 
\begin{equation*}
(\hi + \ci + \Ffunc (\bmax)) / \bmax.
\end{equation*}
The above slopes are sorted in increasing order, and the demand $D_t$ is fulfilled by generators in this order. We initialize $\lambda_t(D_t)$ to be the slope of the last unit used. 

We implemented the state independent Lagrangian approach from \cite{AdM2008} and the DADP approach in MATLAB 64-bit R2014b. For each instance, we used the HTCondor framework to schedule a job on a machine with at least 4 CPUs, 4 GB RAM, and 12 GB disk space. Both the Lagrangian approach and the DADP approach ran on the same machine one after the other for fair time comparisons. Within each instance, for the generator subproblems, we used MATLAB's parpool with 4 workers. We modeled the one-step lookahead DADP MIP described in Subsection \ref{subsection:upper} in GAMS and used the solver CPLEX 12.6. For each of the 500 sample paths, we solved a sequence of MIPs, one for each time period, by using system calls to GAMS from MATLAB.

\subsection{Perfect Information Bound}
\label{subsec:per_info}

For comparison to the DADP approach, we also implemented the perfect information approach that provides a lower bound.
In this approach, we assume perfect knowledge of a sample path (i.e., demand realizations). For a given sample path, the
stochastic unit commitment problem becomes a MIP. We simulated 100 sample paths and averaged the resulting costs to
obtain a stochastic lower bound, and a 95\% confidence interval around this stochastic bound. Note that this
approach only provides a stochastic lower bound and does not generate an implementable policy (and therefore, an
associated upper bound).

We provide the perfect information MIP formulation below. Here again $\zit$ refers to the generation level in time
period $t$, $\uit$ refers to the commitment decision for time period $t+1$, and $\dt$ refers to the demand observed
immediately before determining the generation level $\zit$. To model start up and shut down costs, we introduce
additional turn on variables $w_t^i$, for $i = 1, \ldots, n$ and $t = 1, \ldots, T$, that have shown to result in
stronger relaxations (see \cite{RaT2005}). We model piecewise linear functions with the locally ideal MIP formulation
suggested in \cite{SLL2013}.

\begin{equation*}
\begin{aligned}
& \underset{u, z, \gamma, e, w}{\text{minimize}}
& & \sumn \sumT \left[ e_t^i + \ci \uit + \hi w_t^i \right]  \\
& \text{subject to}
& & \sumn z_t^i = \dt, \; t = 1, \ldots, T, & & \text{(demand satisfaction)} \\
& & & w_0^i = u_0^i, \; i = 1, \ldots, n, \\
& & & w_t^i \geq \uit - u_{t-1}^i, \; i = 1, \ldots, n, \; t = 1, \ldots, T, & & \text{(turn on variables)} \\
& & & \sum_{r=(t' - \liu + 1)^+}^{t'} w_r^i \leq \uit, \; i = 1, \ldots, n, \; t' = 0, \ldots, T - 1, & & \text{(turn on inequalities)} \\
& & & \sum_{r=(t' - \lil + 1)^+}^{t'} w_r^i \leq 1- u_{t-\lil}^i, \; i = 1, \ldots, n, \; t' = \lil, \ldots, T - 1, & & \text{(turn off inequalities)} \\
& & & z_{t-1}^i - \rd - (1-u^i_{t-1}) b_{\min}^i \leq z_t^i, \; i = 1, \ldots, n, \; t = 1, \ldots, T, & & \text{(ramp down constraints)}\\
& & & z_t^i \leq z_{t-1}^i + \ru + w_{t-1}^i b_{\min}^i, \; i = 1, \ldots, n, \; t = 1, \ldots, T, & & \text{(ramp up constraints)}\\
& & & \sum_{k = 0}^{K_i} \gamma_{t, k}^i = u_{t-1}^i, \; i = 1, \ldots, n, \; t = 1, \ldots, T, \\
& & & z_t^i = \sum_{k = 0}^{K_i} b_k^i \gamma_{t, k}^i, \; i = 1, \ldots, n, \; t = 1, \ldots, T, \\
& & & e_t^i = \sum_{k = 0}^{K_i} c_k^i \gamma_{t, k}^i, \; i = 1, \ldots, n, \; t = 1, \ldots, T, & & \text{(PWL cost)}\\
& & & \uit \in \{ 0, 1\}, \; i = 1, \ldots, n, \; t = 0, \ldots, T,\\
& & & w_t^i \in \{ 0, 1\}, \; i = 1, \ldots, n, \; t = 0, \ldots, T,\\
& & & \gamma_{t, k}^i \geq 0, \; k = 0, \ldots, K_i, \; i = 1, \ldots, n, \; t = 1, \ldots, T,\\
\end{aligned}
\end{equation*}
where we define $(x)^+ = \max (0, x)$.

For fair comparison, the perfect information MIP ran after DADP lower bound on the same machine scheduled by HTCondor. The MIP was modeled in GAMS, solved using CPLEX 12.6, and was called through system calls in MATLAB. The MIPs for each sample path ran in parallel through MATLAB's 4 parfor workers. We limited the MIP solver to use a single CPU.

\subsection{Results}

\begin{table}[h]
\caption{Comparison of lower bounds from the state-independent Lagrangian approach, the DADP approach, and the perfect information bound. The upper bounds are generated from the one-step lookahead policy using the value function generated from the DADP approach.}
\label{table:bounds}
\centering
\begin{tabular}{cccrrrrrr}
\toprule
\multirow{2}{*}{$\mu$} & \multirow{2}{*}{$\sigma$} & \multirow{2}{*}{\# Gen}  & \multicolumn{4}{c}{Lower Bounds [\$ millions]} & \multicolumn{2}{c}{Upper Bound [\$ millions]} \\
\cmidrule(r){4-7} \cmidrule(r){8-9}
& & & \multicolumn{1}{c}{$\lambdat$} & \multicolumn{1}{c}{$\lambdatvt$} & \multicolumn{1}{c}{mean$_{PInfo}$} & \multicolumn{1}{c}{HW$_{PInfo}$} & \multicolumn{1}{c}{mean} & \multicolumn{1}{c}{HW}\\
\midrule
0.4 & 0.15 & 15 & 9.17 & 9.65 & 9.48 & 0.03 & 9.94 & 0.02 \\

0.4 & 0.20 & 15 & 9.17 & 10.17 & 9.73 & 0.04 & 10.53 & 0.02 \\

0.4 & 0.25 & 15 & 9.17 & 10.85 & 10.11 & 0.04 & 11.33 & 0.03 \\

0.6 & 0.15 & 15 & 14.96 & 17.75 & 16.89 & 0.08 & 18.08 & 0.05 \\

0.6 & 0.20 & 15 & 14.97 & 19.84 & 18.29 & 0.13 & 20.19 & 0.07 \\

0.6 & 0.25 & 15 & 14.96 & 22.52 & 19.79 & 0.19 & 22.70 & 0.10 \\

0.8 & 0.15 & 15 & 26.99 & 35.04 & 33.76 & 0.31 & 35.80 & 0.14 \\

0.8 & 0.20 & 15 & 26.98 & 39.93 & 37.98 & 0.41 & 41.52 & 0.19 \\

0.8 & 0.25 & 15 & 26.99 & 45.15 & 43.07 & 0.54 & 46.95 & 0.26 \\
\midrule
0.4 & 0.15 & 30 & 7.15 & 8.99 & 8.49 & 0.04 & 9.31 & 0.03 \\

0.4 & 0.20 & 30 & 7.15 & 10.07 & 9.17 & 0.05 & 10.73 & 0.05 \\

0.4 & 0.25 & 30 & 7.15 & 11.17 & 9.91 & 0.07 & 12.53 & 0.08 \\

0.6 & 0.15 & 30 & 11.54 & 15.09 & 13.85 & 0.06 & 15.90 & 0.06 \\

0.6 & 0.20 & 30 & 11.53 & 16.95 & 15.16 & 0.10 & 18.24 & 0.09 \\

0.6 & 0.25 & 30 & 11.54 & 18.90 & 16.63 & 0.16 & 20.76 & 0.12 \\

0.8 & 0.15 & 30 & 16.75 & 23.87 & 22.27 & 0.26 & 25.59 & 0.13 \\

0.8 & 0.20 & 30 & 16.75 & 28.87 & 26.52 & 0.33 & 31.64 & 0.23 \\

0.8 & 0.25 & 30 & 16.75 & 34.12 & 31.48 & 0.46 & 37.50 & 0.22 \\
\midrule
0.4 & 0.15 & 50 & 11.59 & 12.93 & 12.29 & 0.05 & 13.03 & 0.03 \\

0.4 & 0.20 & 50 & 11.59 & 13.93 & 12.85 & 0.06 & 13.96 & 0.05 \\

0.4 & 0.25 & 50 & 11.59 & 15.12 & 13.47 & 0.08 & 15.18 & 0.06 \\

0.6 & 0.15 & 50 & 20.40 & 23.65 & 22.17 & 0.12 & 24.50 & 0.08 \\

0.6 & 0.20 & 50 & 20.40 & 26.45 & 23.79 & 0.19 & 28.15 & 0.15 \\

0.6 & 0.25 & 50 & 20.40 & 30.05 & 26.21 & 0.26 & 31.03 & 0.14 \\

0.8 & 0.15 & 50 & 31.86 & 47.02 & 43.87 & 0.53 & 47.98 & 0.27 \\

0.8 & 0.20 & 50 & 31.85 & 56.42 & 51.97 & 0.82 & 57.81 & 0.43 \\

0.8 & 0.25 & 50 & 31.84 & 67.13 & 62.13 & 1.06 & 68.42 & 0.55 \\

\bottomrule
\end{tabular}
\end{table}

We report the bounds obtained from each of the approaches for the 27 test instances in Table \ref{table:bounds}. For the
lower bounds, $\lambdat$ refers to the state independent Lagrangian bound \cite{AdM2008} and $\lambdatvt$ refers to the
DADP approach. Since the perfect bound is stochastic, we report the average under the column mean$_{PInfo}$ and the
half-width from a 95\% confidence interval under the column HW$_{PInfo}$. It's defined as $\text{HW}_{PInfo} = 1.96 \;
\sigma_{PInfo} / \sqrt{N}$, where $\sigma_{PInfo}$ is the sample standard deviation and $N = 100$ is the number of
sample paths. For the upper bounds, we used the one step lookahead approach with the value function obtained from the
DADP approach. The half-width is again defined similarly except we used $N = 500$ sample paths. We see that the DADP approach provides improved lower bounds over the perfect information and state-independent approach. The upper bounds show that we are not too far from closing the optimality gap.

In Table \ref{table:times}, we report the solve times for the lower bounds. The results in each row were obtained on the
same machine, so comparing solve times comparisons between approaches within each row are meaningful. However, different instances may have been
run on different machines, so we should not compare solve times in different rows to each other. For most instances, the
state-independent approach is faster than the DADP and perfect information approaches. The DADP approach is slower than
the perfect information bound for the 15 generator case, but faster for the 30 generator instances. For the 50 generator
instances, it was surprising that in comparison to the DADP approach the perfect information approach had a
comparable speed for most instances and was even faster for a few of them. After looking into this further, we found that
the 30 generator instances had a different mix of generators than the 50 generator instances, which made them more
difficult to solve. In particular, a few of the generators were long term generators that had large minimum and maximum
generation levels and once turned on had to remain on for the remainder of the time horizon. Overall, the DADP approach
provides better bounds, and scales similarly and possibly better than the perfect information bounds. For the one-step lookahead upper bound, we solved a MIP for each time
period for each of the 500 sample paths. The average solve time for each MIP in the first 50 generator instance was
about 1.1 seconds, indicating that this may be practical for implementation.

\begin{table}[h]
\caption{Solve times for the state-independent Lagrangian approach, the DADP approach, and the perfect information bound reported in minutes.}
\label{table:times}
\centering
\begin{tabular}{cccrrr}
\toprule
\multirow{2}{*}{$\mu$} & \multirow{2}{*}{$\sigma$} & \multirow{2}{*}{\# Gen}  & \multicolumn{3}{c}{Solve Time [min]} \\
\cmidrule(r){4-6}
& & & \multicolumn{1}{c}{$\lambdat$} & \multicolumn{1}{c}{$\lambdatvt$} & \multicolumn{1}{c}{PInfo} \\
\midrule
0.4 & 0.15 & 15 & 3.2 & 40.8 & 3.7 \\

0.4 & 0.20 & 15 & 7.6 & 151.5 & 7.3 \\

0.4 & 0.25 & 15 & 7.6 & 150.9 & 12.8 \\

0.6 & 0.15 & 15 & 7.9 & 150.5 & 4.8 \\

0.6 & 0.20 & 15 & 3.2 & 40.4 & 2.4 \\

0.6 & 0.25 & 15 & 12.8 & 99.4 & 7.8 \\

0.8 & 0.15 & 15 & 5.3 & 88.4 & 3.9 \\

0.8 & 0.20 & 15 & 7.8 & 137.1 & 3.7 \\

0.8 & 0.25 & 15 & 8.5 & 136.4 & 3.8 \\
\midrule
0.4 & 0.15 & 30 & 20.3 & 137.9 & 986.2 \\

0.4 & 0.20 & 30 & 51.6 & 341.1 & 1264.1 \\

0.4 & 0.25 & 30 & 66.1 & 316.3 & 1120.5 \\

0.6 & 0.15 & 30 & 66.5 & 315.1 & 480.5 \\

0.6 & 0.20 & 30 & 67.1 & 316.2 & 1171.7 \\

0.6 & 0.25 & 30 & 20.0 & 158.4 & 1379.1 \\

0.8 & 0.15 & 30 & 20.2 & 160.3 & 151.7 \\

0.8 & 0.20 & 30 & 48.1 & 347.1 & 435.7 \\

0.8 & 0.25 & 30 & 47.9 & 347.9 & 515.9 \\
\midrule
0.4 & 0.15 & 50 & 34.8 & 271.7 & 225.9 \\

0.4 & 0.20 & 50 & 21.8 & 249.4 & 263.5 \\

0.4 & 0.25 & 50 & 21.6 & 248.9 & 281.6 \\

0.6 & 0.15 & 50 & 22.2 & 250.0 & 174.2 \\

0.6 & 0.20 & 50 & 18.9 & 256.5 & 101.1 \\

0.6 & 0.25 & 50 & 23.9 & 281.7 & 58.7 \\

0.8 & 0.15 & 50 & 10.3 & 134.6 & 8.8 \\

0.8 & 0.20 & 50 & 22.6 & 250.7 & 21.1 \\

0.8 & 0.25 & 50 & 22.3 & 253.7 & 20.9 \\
\bottomrule
\end{tabular}
\end{table}

\section{Conclusion}

In our numerical results, we included a single coupling demand constraint. If there are multiple loads, we could in principle have a coupling constraint for each one and relax each a different sets of multipliers. However, this would create further variables to optimize the Lagrangian function, and the obtained bounds may become weaker. Future work could address handling the case of multiple coupling constraints efficiently.

In this paper, we assumed that the demands were independent from one time period to another. A simple extension, such as having weather states that indicate the demand distribution, is possible. However, more complex modeling such as a two-level Markov model with hidden states \cite{NJD2017} would require further investigation.

\section{Acknowledgements}

Thanks to Dr. Yanchao Liu for providing the FERC generator data.   
This work is supported by the U.S. Department of Energy, Office of
Science, Office of Advanced Scientific Computing Research, Applied Mathematics program under contract number
DE-AC02-06CH11357.

\bibliographystyle{abbrv}
\bibliography{WeaklyCoupledPaper}
\end{document}